\documentclass[11pt, amsfonts]{amsart}

\usepackage{amsmath,amssymb,amsthm,color,enumerate,mathrsfs}
\usepackage[all]{xy}

\usepackage[usenames]{xcolor}
\definecolor{refkey}{rgb}{0.6, 0.7, 0.4}
\definecolor{labelkey}{rgb}{0, 0.7, 0.5}
\usepackage{todonotes}

\SelectTips{eu}{12}

\newtheorem{theorem}{Theorem}[section]
\newtheorem{proposition}[theorem]{Proposition}

\newtheorem{corollary}[theorem]{Corollary}

\theoremstyle{definition}
\newtheorem{definition}[theorem]{Definition}

\newtheorem{question}[theorem]{Question}

\theoremstyle{remark}

\numberwithin{equation}{section}

\newcommand{\II}{\mathfrak{I}}

\DeclareMathOperator{\Bd}{Bd}
\DeclareMathOperator{\Cyl}{Cyl}

\makeatletter
\@namedef{subjclassname@2020}{%
  \textup{2020} Mathematics Subject Classification}
\makeatother

\title[Simple homotopy and contractible contractions]{Simple homotopy of flag simplicial complexes and contractible contractions of graphs}

\author[A. Dochtermann]{Anton Dochtermann}
\address{Department of Mathematics, Texas State University, San Marcos, TX 78666, USA}
\email{dochtermann@txstate.edu}

\author[T. Matsushita]{Takahiro Matsushita}
\address{Department of Mathematical Sciences, Faculty of Science, Shinshu University, Matsumoto, Nagano 390-8621, Japan}
\email{matsushita@shinshu-u.ac.jp}

\keywords{clique complex, flag complex, simple homotopy}
\subjclass[2020]{57Q10; 05E45}

\begin{document}
\maketitle

\begin{abstract}
In his work on molecular spaces, Ivashchenko introduced the notion of an $\II$-contractible transformation on a graph $G$, a family of addition/deletion operations on its vertices and edges. Chen, Yau, and Yeh used these operations to define the $\II$-homotopy type of a graph, and showed that $\II$-contractible transformations preserve the simple homotopy type of $C(G)$, the clique complex of $G$. In other work, Boulet, Fieux, and Jouve introduced the notion of $s$-homotopy of graphs to characterize the simple homotopy type of a flag simplicial complex. They proved that $s$-homotopy preserves $\II$-homotopy, and asked whether the converse holds. In this note, we answer their question in the affirmative, concluding that graphs $G$ and $H$ are $\II$-homotopy equivalent if and only if $C(G)$ and $C(H)$ are simple homotopy equivalent. We also show that a finite graph $G$ is $\II$-contractible if and only if $C(G)$ is contractible, which answers a question posed by the first author, Espinoza, Fr\'ias-Armenta, and Hern\'andez. We use these ideas to give a characterization of simple homotopy for arbitrary simplicial complexes in terms of links of vertices.
\end{abstract}

\section{Introduction}\label{sec:intro}

Simple homotopy equivalence is a combinatorial refinement of homotopy equivalence based on elementary collapses and expansions of cell complexes. We recall relevant definitions in the next section, and refer to the textbook \cite{Cohen} for a complete treatment. A simplicial complex $X$ is \emph{flag} if any minimal non-face has cardinality at most two. Such a complex is determined by its $1$-skeleton (the subcomplex consisting of 0-dimensional and 1-dimensional faces), and thus can be identified with a simple graph $G = G(X)$.  On the other hand, if $G$ is any finite simple graph one may construct its \emph{clique complex} $C(G)$, the flag simplicial complex whose faces are the complete subgraphs of $G$. 
Flag complexes are particularly efficient to encode and are well-suited for calculations in topological data analysis and other applied settings.
With this in mind it is of interest to understand how simple homotopy of flag simplicial complexes can be understood in graph theoretic terms.

In this note we study how simple homotopy equivalence of flag complexes relates to the \emph{$\II$-homotopy} type of a graph, a notion studied by Chen, Yau, Yeh \cite{CYY}. The theory is based on the \emph{$\II$-contractible transformations} of Ivashchenko \cite{Ivashchenko1}, who introduced these operations on graphs in his study of molecular spaces. In \cite{Ivashchenko3}, Ivashchenko showed that such transformations do not change the homology groups of the underlying clique complexes.  Since then, the theory has been further explored by several authors \cite{EFH, Frias, Ivashchenko2, GG, ZWZ}, and in \cite{CYY} Chen, Yau, and Yeh \cite{CYY} showed that $\II$-contractible transformations on a graph do not change the simple homotopy type of the underlying clique complexes.

On the other hand, in \cite{BFJ} Boulet, Fieux, and Jouve introduced a notion of \emph{$s$-homotopy} for graphs as a way to study simple homotopy of flag complexes. Their construction was inspired by the \emph{weak beat points} of finite spaces introduced by Barmak and Minian in \cite{BM}. Boulet et al. showed that if $G$ and $H$ are $s$-homotopy equivalent then $G$ and $H$ are $\II$-homotopy equivalent. They asked whether the converse holds.

\begin{question}[Section 5 of \cite{BFJ}] \label{question BFJ}
Let $G$ be a finite simple graph. Is it true that the $s$-homotopy type of $G$ and the $\II$-homotopy type of $G$ agree?
\end{question}

The first contribution of the present paper is to give an affirmative answer to this question. In fact we establish the following list of equivalences.

\begin{theorem} \label{thm:main}
Suppose $G$ and $H$ are finite graphs.  Then the following are equivalent.

\begin{enumerate}
    \item The graphs $G$ and $H$ are $\II$-homotopy equivalent;
    \item The graphs $G$ and $H$ are $s$-homotopy equivalent;
    \item The clique complexes $C(G)$ and $C(H)$ of $G$ and $H$ are simple homotopy equivalent.
\end{enumerate}
\end{theorem}

In \cite{DEFH}, the first author, Espinoza, Fr\'ias-Armenta, and Hern\'andez posed a related question.

\begin{question}[Question 28 of \cite{DEFH}]
Does there exist a finite graph $G$ that is not $\II$-contractible, but such that $C(G)$ is contractible?
\end{question}

The second main result in this note is to give a negative answer to this question.

\begin{theorem} \label{thm:other}
Suppose $G$ is a simple graph such that $C(G)$ is contractible. Then $G$ is $\II$-contractible. 
\end{theorem}


Our proofs of Theorems \ref{thm:main} and \ref{thm:other} are quite straightforward, and follow from a combination of known results from the literature. It seems that the connection among various approaches to graph homotopy is not well known even to experts in the field, and one of our goals in this work is to have this spelled out more clearly.

It is known that a finite simplicial complex is contractible if and only if it is simple homotopy equivalent to a point (see Theorem \ref{theorem simple homotopy} for the more general statement). From this we obtain the following corollary.


\begin{corollary} \label{main corollary 1}
The class $\II$ coincides with the class of finite graphs whose clique complexes are contractible.
\end{corollary}


Since graphs and flag simplicial complexes are in one-to-one correspondence, we can translate notions between the two. From this perspective, Corollary 1.5, when expressed in the language of flag complexes, implies that $\II$-contractible transformations correspond to operations of adding or deleting vertices or edges with contractible links (see Definition \ref{definition contractible transformation}).

In \cite{CYY}, it was shown that the operations of edge addition and deletion ($(\II3)$ and $(\II4)$ in Definition \ref{definition contractible transformation}) can in fact each be realized as a combination of the vertex additions and deletions ($(\II1)$ and $(\II2)$ in Definition \ref{definition contractible transformation}). Hence Corollary \ref{main corollary 1} implies that two flag complexes are simple homotopy equivalent if and only if one is obtained from the other by adding or removing vertices with contractible links (Corollary \ref{main corollary 3}). It is natural to ask whether a similar assertion holds for general simplicial complexes. Utilizing ideas from \cite{BFJ}, we are able to show that this is indeed true.

\begin{theorem} \label{thm:simplehom}
Two simplicial complexes $K$ and $L$ are simple homotopy equivalent if and only if there exists a sequence
\[ K = X_0, X_1, \cdots, X_k = L \]
of simplicial complexes such that $X_i$ is obtained from $X_{i-1}$ by adding or deleting a vertex that has a contractible link.
\end{theorem}


The rest of this paper is organized as follows. We review the necessary background in Section \ref{sec:prelim} and provide the proofs of Theorems \ref{thm:main}, \ref{thm:other}, and \ref{thm:simplehom} in Section \ref{sec:proof}.

\section{Preliminaries}\label{sec:prelim}
We begin with a review of the relevant definitions and prior work needed to state and prove our results. We assume the reader is familiar with the basics of graphs and simplicial complexes. For a comprehensive introduction to the combinatorial topology of simplicial complexes, we refer the reader to \cite{Koz}.

\subsection{Simple homotopy}
Suppose $X$ is a simplicial complex and $\tau \subset \sigma$ are faces such that $\sigma$ is a facet of $X$, and no other facet of $X$ contains $\tau$. Then $\tau$ is a \emph{free face} of $X$ and a \emph{simplicial collapse} of $X$ is the removal of all simplices $\alpha$ satisfying $\tau \subset \alpha \subset \sigma$. If $\dim \tau = \dim \sigma - 1$ then we call this removal an \emph{elementary collapse}. An \emph{elementary expansion} is the inverse of an elementary collapse. 

\begin{definition}
Two simplicial complex $X$ and $Y$ are \emph{simple homotopy equivalent} if there exists a sequence
\[ X = X_0, \cdots, X_n = Y\]
such that $X_i$ is an elementary collapse or an elementary expansion of $X_{i-1}$.
\end{definition}


A homotopy equivalence $f:X \rightarrow Y$ is a \emph{simple homotopy equivalence} if it is homotopic to a sequence of collapses and expansions.   An important resut in the area says that a homotopy equivalence $f:X \rightarrow Y$ between finite complexes is a simple homotopy equivalence if and only if the \emph{Whitehead torsion} $\tau(f)$ of $f$ vanishes \cite[22.2]{Cohen}. If $\pi_1(X) \cong \pi_1(Y)$ is trivial it is not hard to see that $\tau(f)=0$. As a consequence we get the following result.

\begin{theorem}\label{theorem simple homotopy}
Let $X$ and $Y$ be simply connected finite simplicial complexes. If $X$ and $Y$ are homotopy equivalent, then $X$ and $Y$ are simple homotopy equivalent.
\end{theorem}

\subsection{$s$-homotopy}
We next recall the notion of $s$-homotopy of graphs as developed in \cite{BFJ}. Here we work with finite simple graphs consisting of a vertex set $V = V(G)$ and edge set $E=E(G)$, and we do not allow loops nor multiple edges. For a graph $G$ and vertex $v \in V(G)$ we use $N_G(v)$ to denote the \emph{(open) neighborhood} of $v$, the set of vertices in $G$ that are adjacent to $v$. We let $N_G(v,w) := N_G(v) \cap N_G(w)$ and $N_G[v] := N_G(v) \cup \{v\}$ for vertices $v$ and $w$ in $G$.  For a subset $S \subset V(G)$ of vertices in $G$, we let $G[S]$ denote the subgraph of $G$ induced by $S$. For a vertex $v$ of $G$, we write $G - v$ to indicate $G[V(G) - \{ v \}]$. For an edge $e$ of $G$, we write $G \setminus e$ to mean the subgraph $V(G \setminus e) = V(G)$ and $E(G \setminus e) = E(G) - \{ e\}$.

A vertex $v \in V(G)$ is \emph{dismantlable} if there exists a vertex $w \neq v$ such that $N_G[v] \subset N_G[w]$. A graph $G$ is said to be \emph{dismantlable} if it can be reduced to a single vertex by removing dismantlable vertices. Dismantlable graphs have applications in many areas of combinatorics including pursuit-evasion games, homomorphism reconfiguration, and statistical physics. Motivated by constructions of \cite{BM} Boulet, Fieux, and Jouve defined the following in \cite{BFJ}.

\begin{definition}
A vertex $v$ of a graph $G$ is \emph{$s$-dismantlable} if $G[N_G(v)]$ is dismantlable.
\end{definition}

If $v$ is an $s$-dismantlable vertex of $G$, we call $G - v$ an \emph{$s$-collapse of $G$}. A graph $H$ is an \emph{$s$-expansion of $G$} if $G$ is an $s$-collapse of $H$.


\begin{definition}
Two graphs $G$ and $H$ are \emph{$s$-homotopy equivalent} if there exists a sequence
\[ G = G_0, \cdots, G_n = H\]
such that $G_i$ is obtained from $G_{i-1}$ by adding or deleting an $s$-dismantlable vertex.
\end{definition}

Boulet, Fieux, and Jouve showed that $s$-homotopy characterizes simple homotopy of flag simplicial complexes as follows.

\begin{theorem} \cite[Theorem 2.10]{BFJ} \label{theorem BFJ}
Suppose $G$ and $H$ are finite graphs. Then $G$ and $H$ are $s$-homotopy equivalent if and only if the clique complexes $C(G)$ and $C(H)$ are simple homotopy equivalent.
\end{theorem}

\subsection{$\II$-contractible transformations}
Motivated by operations in the theory of molecular spaces and digital topology,  Ivashchenko \cite{Ivashchenko1} introduced notions of {\it $\II$-contractible transformations} that can be performed on a graph. To describe these transformations we first define the class of $\II$-contractible graphs. 


\begin{definition} \label{definition class I}
The class of \emph{$\II$-contractible graphs}, denoted by $\II$, is the smallest class of finite graphs that contains $K_1$ (a single vertex), and is closed under the following operations.
\begin{enumerate}

\item If $G \in \II$ and $G[N_G(v)] \in \II$ for some vertex $v \in V(G)$, then $G - v \in \II$.

\item If $G$ is a graph such that $G - v \in \II$ and $G[N_G(v)] \in \II$ for some $v \in V(G)$, then $G \in \II$.

\item If $G \in \II$ and $G[N_G(v, w)] \in \II$ for some edge $e = \{v,w\} \in E(G)$, then $G \setminus e \in \II$.

\item  If $G$ is a graph such that $G \setminus e \in \II$ and $N_G(v, w) \in \II$ and for some $e = \{v,w\} \in E(G)$, then $G \in \II$.
\end{enumerate}

\end{definition}

\begin{definition} \label{definition contractible transformation}
An \emph{$\II$-contractible transformation} on a graph $G$ is any of the following.
\begin{enumerate}[$(\II 1)$]
\item Deleting a vertex: If $v$ is a vertex of $G$ such that $G[N_G(v)] \in \II$, then $v$ can be deleted from $G$.

\item Gluing a vertex: If $S \subset V(G)$ such that $G[S] \in \II$, then one can add a new vertex $v$ to $G$ satisfying $N_G(v) =S$.

\item Deleting an edge: If $e = \{ v, w\}$ is an edge of $G$ satisfying $G[N_G(v, w)] \in \II$, then $e$ can be deleted from $G$.

\item Gluing an edge: Suppose $e = \{ v, w\}$ is a $2$-element subset of $V(G)$ which is not an edge in $G$. If $G[N_G(v, w)] \in \II$ then one can add the edge $\{v,w\}$.
\end{enumerate}
\end{definition}

Allowing sequences of such transformations, we define an equivalence relation on the set of finite graphs. Namely, we say that graphs $G$ and $H$ are \emph{$\II$-homotopy equivalent} if there exists a sequence
\[ G = G_0 , G_1, \cdots, G_n = H\]
such that $G_i$ is obtained from $G_{i-1}$ by one of the operations $(\II 1)$, $(\II 2)$, $(\II 3)$, and $(\II 4)$. The class $\II$ can be seen to coincide with the class of finite simple graphs which are $\II$-homotopy equivalent to a single vertex $K_1$.


In \cite{Ivashchenko3} it was shown that $\II$-contractible transformations on a graph $G$ preserve the homology groups of its clique complex $C(G)$.  Chen, Yau, and Yeh proved that they in fact preserve the \emph{simple homotopy type} of $C(G)$.

\begin{theorem}\cite[Theorem 3.7]{CYY}\label{theorem CYY}
If two graphs $G$ and $H$ are $\II$-homotopy equivalent, then their clique complexes $C(G)$ and $C(H)$ are simple homotopy equivalent.
\end{theorem}

Boulet, Fieux, and Jouve were also able to relate the notion of $\II$-homotopy to $s$-homotopy. 

\begin{proposition}\cite[Proposition 5.1]{BFJ} \label{proposition BFJ}
If two graphs $G$ and $H$ have the same $s$-homotopy type, then they have the same $\II$-homotopy type.
\end{proposition}

\section{Proofs and consequences}\label{sec:proof}
Our proofs of Theorems \ref{thm:main} and \ref{thm:other} are now straightforward applications of the results discussed above.

\begin{proof}[Proof of Theorem \ref{thm:main}]
The equivalence $(2) \Leftrightarrow (3)$ follows from Theorem \ref{theorem BFJ}, the implication $(1) \Rightarrow (3)$ follows from Theorem \ref{theorem CYY}, and the implication $(2) \Rightarrow (1)$ follows from Proposition \ref{proposition BFJ}. The result follows.
\end{proof}


\begin{proof}[Proof of Theorem \ref{thm:other}]
Suppose $G$ is a finite graph such that $C(G)$ is contractible. Since $C(G)$ is homotopy equivalent to a point, from Theorem \ref{theorem simple homotopy} 
we have that $C(G)$ is simple homotopy equivalent to a point. Hence Theorem \ref{thm:main} implies that $G$ is $\II$-contractible.
\end{proof}

\subsection{Other applications}

 In \cite{CYY} Chen, Yau, and Yeh showed that there is redundancy in Definition \ref{definition contractible transformation} of $\II$-contractible transformations. Namely, they showed that the operations $(\II 3)$ and $(\II 4)$ can be obtained from a combination of the operations $(\II 1)$ and $(\II 2)$. Combining this with our results then implies the following.

\begin{corollary} \label{main corollary 2}
Suppose $G$ and $H$ are finite graphs. Then $C(G)$ and $C(H)$ are simple homotopy equivalent if and only if there exists a sequence
\[ G = G_0, \cdots, G_n = H\]
of graphs such that $G_i$ is obtained from $G_{i-1}$ by adding or deleting a vertex $v$ such that $C(G_i[N(v)])$ is contractible.
\end{corollary}

We can also translate these statements into the language of flag simplicial complexes. In what follows recall that if $X$ is a simplicial complex and $v \in V(X)$, the link of $v$ is the simplicial complex 
\[{\rm link}_X(v) = \{F \in X ~|~
v \notin F,~~ F \cup \{v\} \in X\}.\]
\noindent
Note that if $X = C(G)$ is a flag simplicial complex then 
${\rm link}_X(v) = C(G[N(v)])$. With this formulation we have the following.

\begin{corollary} \label{main corollary 3}
Suppose $X$ and $Y$ are flag simplicial complexes. Then $X$ and $Y$ are simple homotopy equivalent if and only if there exists a sequence
\[ X = X_0, \cdots, X_n = Y\]
\noindent
of flag complexes such that $X_i$ is obtained from $X_{i-1}$ by adding or deleting a vertex whose link is contractible.
\end{corollary}

As mentioned in the introduction, Corollary \ref{main corollary 3} holds for not necessarily flag simplicial complexes. We end with a proof of this result.


\begin{proof}[Proof of Theorem \ref{thm:simplehom}]
For simplicial complexes $K$ and $L$ we write $K \sim L$ if there exists a sequence described above. If such a sequence exists, then 20.1 or 5.9 in \cite{Cohen} shows that $K$ and $L$ are simple homotopy equivalent. 

For the converse, we first show that $K \sim \Bd(K)$. Here $\Bd(K)$ denotes the \emph{barycentric subdivision} of $K$, the order complex of the face poset of $K$. To see that $K \sim \Bd(K)$ we follow the strategy employed by Boulet et all in the proof of Proposition 2.9 from \cite{BFJ}. Namely, we consider the simplicial complex $\Cyl(K)$ defined as follows. The vertex set of $\Cyl(K)$ is the disjoint union $V(K) \sqcup V(\Bd(K))$, and a subset $\sigma$ of $V(\Cyl(K))$ is a simplex of $\Cyl(K)$ if the following conditions are satisfied:
\begin{enumerate}[(1)]
\item $V(K) \cap \sigma \in K$;

\item $V(\Bd(K)) \cap \sigma \in \Bd(K)$;

\item $\alpha \in \sigma \cap V(\Bd(K))$ implies $\sigma \cap V(K) \subset \alpha$.
\end{enumerate}
Note that $K$ and $\Bd(K)$ are subcomplexes of $\Cyl(K)$. The construction of $\Cyl(K)$ is similar to the construction of the graph $H$ in the proof of Proposition 2.9 in \cite{BFJ}. By mimicking their arguments, we can see that $K \sim \Cyl(K) \sim \Bd(K)$.

Now suppose that $K$ and $L$ are simple homotopy equivalent. We would like to show that $K \sim L$. It is known that $K$ and $\Bd(K)$ are simple homotopy equivalent (see Proposition 6.18 of \cite{Koz}), and hence that $\Bd(K)$ and $\Bd(L)$ are simple homotopy equivalent. Thus Corollary \ref{main corollary 3} implies that $\Bd(K) \sim \Bd(L)$. This shows that $K \sim \Bd(K) \sim \Bd(L) \sim L$, which completes the proof.
\end{proof}







\section*{Acknowledgment}
The first author was partly supported by Simons Foundation Grant $\#964659$. The second author was partly supported by JSPS KAKENHI Grant Numbers JP19K14536 and JP23K12975. We thank Mart\'in Eduardo Fr\'ias Armenta and Shuichi Tsukuda for helpful discussions.

\end{document}